\documentclass[a4paper,11pt]{article}
\usepackage{amsfonts}
\usepackage{amssymb}
\usepackage{amsmath}
\usepackage{amsthm}
\usepackage{verbatim}
\usepackage{algorithmic}
\usepackage[T1]{fontenc}
\theoremstyle{plain}
\newtheorem{theorem}{Theorem}[section]

\newtheorem{corollary}[theorem]{Corollary}
\newtheorem{proposition}[theorem]{Proposition}
\theoremstyle{definition}
\newtheorem{definition}[theorem]{Definition}

\begin{document}
\title{First-order logic with incomplete information}
\author{Antti Kuusisto\\
{\small Tampere University, University of Bremen}}

\date{}

\maketitle

\begin{abstract}
\noindent
We develop first-order logic and some extensions for incomplete information scenarios
and consider related complexity issues.
\end{abstract}

\section{Introduction}
We define (an extension of) first-order logic for
scenarios where the underlying model is not 
fully known. This is achieved by
evaluating a formula 
with respect to several models simultaneously,
not unlike in first-order modal logic.
The set (or even a proper class) of models is taken to
represent a collection of all possible models.
The approach uses some ingredients 
from Hodges' team semantics.
We shall not formally define what we mean by
incomplete information (or imperfect information for
that matter). However, we will not directly
investigate any variant of quantifier independence as in IF-logic
(which is sometimes referred to as first-order logic
with \emph{imperfect} information).
To demonstrate the defined framework from a
technical perspective we also provide a
complexity (of satisfiability) result that can be
easily extended to further similar systems not
formally studied here.

\section{First-order logic with incomplete information}

%
%Let $\mathcal{M}$ be a set (or even a class) of models
%
Let $\tau$ be a relational signature. Let $F(\tau)$ be
the smallest set such that the following conditions hold.
\begin{enumerate}
\item
For any $R\in\tau$, $Rx_1...x_k\in F(\tau)$.
Here $x_1,...,x_k$ are arbitrary variables (with 
possible repetitions) from a fixed countably infinite
set $\mathrm{VAR}$ of first-order variable symbols. $R$ is a $k$-ary
relation symbol.
\item
$x=y\in F(\tau)$ for all $x,y\in\mathrm{VAR}$.
\item
If $\varphi,\varphi'\in F(\tau)$, then $(\varphi\wedge\varphi')\in F(\tau)$.
\item
If $\varphi\in F(\tau)$, then $\neg\varphi\in F(\tau)$.
\item
If $x\in\mathrm{VAR}$ and $\varphi\in F(\tau)$, then $\exists x\varphi\in F(\tau)$.
%
%
%
%\item
%If $x\in\mathrm{VAR}$ and $\varphi\in F(\tau)$, then $C x\varphi\in F(\tau)$.
%
%
%
\end{enumerate}
The above defines the exact syntactic version of first-order logic we
shall consider here.
The semantics of (this version of) first-order logic is here defined
with respect to \emph{$\tau$-interpretation classes}; a $\tau$-interpretation is a
pair $(\mathfrak{M},f)$ where $\mathfrak{M}$ is a $\tau$-model
and $f$ a finite function that maps a finite set of variable symbols
into the domain of $\mathfrak{M}$. 
A $\tau$-interpretation class is a
set (or a class) of $\tau$-interpretations with the functions $f$
having the same domain.
From now on we will only consider $\tau$-interpretation
classes that are sets and call these classes \emph{model sets}; we
acknowledge that a pair $(\mathfrak{M},f)$ is more 
than a model due to the function $f$, and indeed such pairs $(\mathfrak{M},f)$
are often called interpretations (while $f$ is an assignment).
Having acknowledged this issue, we shall not dwell on it any more,
and we shall even occasionally call pairs $(\mathfrak{M},f)$ models.
We note that a model set could also be called a \emph{model team} or
even an \emph{unknown model} (in singular indeed).
A \emph{choice
function} for a model set $\mathcal{M}$ is a 
function that maps each model $(\mathfrak{M},f)$ in $\mathcal{M}$ to
some element $a$ in the domain of $\mathfrak{M}$.
Recall that $h[a/b]$ denotes the function $h$
modified or extended so that $b$ maps to $a$.
If $F$ is a choice function, we let $\mathcal{M}[F/x]$ denote
the class
\begin{equation*}
\{(\mathfrak{M},f[F(\mathfrak{M},f)/x])
\, |\, (\mathfrak{M},f)\in\mathcal{M}\}.
\end{equation*}
We let $\mathcal{M}[\top/x]$ denote
the class
\begin{equation*}
\{(\mathfrak{M},f[b/x])
\, |\, (\mathfrak{M},f)\in\mathcal{M}\text{ and }b\in\mathit{Dom}(\mathfrak{M})\}.
\end{equation*}
The \emph{common domain} of a model set $\mathcal{M}$ is
the (possibly empty) intersection of the domains of the models in $\mathcal{M}$. If $A$
is any subset (including the empty set) of the common domain of $\mathcal{M}$,
we let $\mathcal{M}[A/x]$ denote
the class
\begin{equation*}
\{(\mathfrak{M},f[b/x])
\, |\, (\mathfrak{M},f)\in\mathcal{M}\text{ and }b\in A\}.
\end{equation*}
Recall that a \emph{constant function} is a function
that maps each input to the same element.
Thus a constant choice function for a model set $\mathcal{M}$ is a
choice function that 
maps each model to the 
same element in the intersection of the
domains of the models in $\mathcal{M}$).
(The empty function is not a constant choice function for
any other than the empty model set.)
Let $\mathcal{M}$ be a $\tau$-interpretation class, i.e., a model set.
The semantics of first-order logic (with incomplete information) is defined as follows.

\[
\begin{array}{ll}
\mathcal{M}\models^+ Rx_1...x_k\
&\text{ iff }\ (f(x_1),...,f(x_k))\in R^{\mathfrak{M}}\text{ for all }
(\mathfrak{M},f)\in\mathcal{M}\\
\mathcal{M}\models^+ (\varphi\wedge\psi)
&\text{ iff }\ \mathcal{M}\models^+ \varphi
\text{ and }\mathcal{M}\models^+ \psi\\
\mathcal{M}\models^+ \neg\varphi
&\text{ iff }\ \mathcal{M}\models^-\varphi\\
\mathcal{M}\models^+ \exists x \varphi
&\text{ iff }\ \mathcal{M}[F/x]\models^+ \varphi
\text{ for some choice}\\
&$ $\ \ \ \ \ \text{function for }\mathcal{M}\\
\mathcal{M}\models^- Rx_1...x_k\
&\text{ iff }\ (f(x_1),...,f(x_k))\not\in R^{\mathfrak{M}}\text{ for all }
(\mathfrak{M},f)\in\mathcal{M}\\
\mathcal{M}\models^- (\varphi\wedge\psi)
&\text{ iff }\ \mathcal{M'}\models^- \varphi
\text{ and }\mathcal{M''}\models^- \psi
\ \ \ \text{ for some }\\
&\text{ }\ \ \ \ \, \mathcal{M',M''}\text{ s.t. }\mathcal{M'}
\cup\mathcal{M''} = \mathcal{M}.\\
\mathcal{M}\models^- \neg\varphi
&\text{ iff }\ \mathcal{M}\models^+\varphi\\
\mathcal{M}\models^- \exists x \varphi
&\text{ iff }\ \mathcal{M}[\top/x]\models^- \varphi\\
%
%\text{ for every choice}\\
%
%&$ $\ \ \ \ \ \text{function for }\mathcal{M}\\
%
%
%
%\mathcal{M}\models^- Cx \varphi
%
%&\text{ iff }\ \mathcal{M}[\cap/x]\models^- \varphi
%
%\text{ for every constant}\\
%
%&$ $\ \ \ \ \ \text{choice function for }\mathcal{M}
%
%
%
\end{array}
\]
Technically this logic (first-order logic
with incomplete information) adds very little to standard
first-order logic: the semantics has simply been
lifted to the level of \emph{sets} of models (or
sets of pairs $(\mathfrak{M},f)$), as
the following Proposition shows.
However, conceptually the 
difference with standard first-order logic
approach is clear, and further meaningful divergence 
can be expected to arise in the study of
extensions of this base formalism.
The following proposition is easy to prove.
\begin{proposition}\label{fok}
Let $\varphi$ be an $\mathrm{FO}$-formula.
Then we have
\begin{itemize}
\item
$\mathcal{M}\models^+\varphi$ iff
$(\mathfrak{M},f)\models_{\mathrm{FO}}\varphi$
for all $(\mathfrak{M},f)\in\mathcal{M}$,
\item
$\mathcal{M}\models^-\varphi$ iff
$(\mathfrak{M},f)\not\models_{\mathrm{FO}}\varphi$ for
all $(\mathfrak{M},f)\in\mathcal{M}$.
\end{itemize}
\end{proposition}
\begin{corollary}\label{cor}
Let $\varphi$ be an $\mathrm{FO}$-formula.
Then
\begin{itemize}
\item
$\{(\mathfrak{M},f)\}\models^+\varphi$ iff
$(\mathfrak{M},f)\models_{\mathrm{FO}}\varphi$,
\item
$\{(\mathfrak{M},f)\}\models^-\varphi$ iff
$(\mathfrak{M},f)\not\models_{\mathrm{FO}}\varphi$.
\end{itemize}
\end{corollary}
We then extend the above 
defined syntax for first-order logic by a
formula construction rule $\varphi \mapsto Cx\varphi$.
We call the resulting language $L_C^*$.
We let $L_C$ be the fragment of $L_C^*$
where $Cx$ is not allowed in the scope of
negation operators. 
We extend the semantics based on model sets as follows,
where by a constant choice function we mean a
choice function that sends all inputs to the same (existing) element.
\[
\begin{array}{ll}
\mathcal{M}\models^+ Cx \varphi
&\text{ iff }\ \mathcal{M}[F/x]\models^+ \varphi
\text{ for some constant}\\
&$ $\ \ \ \ \ \text{choice function $F$ for }\mathcal{M}
\end{array}
\]
The reading of the operator $Cx$ could be something in the 
lines of there existing a \emph{common $x$}, or perhaps a \emph{shared} or 
\emph{constant} $x$, or even \emph{known} or
\emph{constructible} $x$.
The above suffices for $L_C$.
To define a (possible) semantics for $L_C^*$, we give
the following clause, where $M$ denotes the
common domain of $\mathcal{M}$.
\[
\begin{array}{ll}
\mathcal{M}\models^- Cx \varphi
&\text{ iff }\ \mathcal{M}[M/x]\models^- \varphi
\end{array}
\]
We shall discuss $L_C^*$ somewhat little as it is
somewhat harder to interpret intuitively than $L_C$.
Let us say that two
formulae $\varphi,\varphi'\in L_C^*$ are \emph{existential variants} if $\varphi$
can be obtained from $\varphi'$ by replacing some (possibly none) of the 
quantifiers $\exists x$ by $Cx$
and some (possibly none) of the quantifiers $Cx$ by $\exists x$.
The following is easy to prove (cf. Corollary \ref{cor}).
\begin{proposition}\label{okman}
Let $\varphi$ be an $\mathrm{FO}$-formula
and assume $\varphi'\in L_C^*$ is an existential variant of $\varphi$. Then
\begin{itemize}
\item
$\{(\mathfrak{M},f)\}\models^+\varphi'$ iff
$(\mathfrak{M},f)\models_{\mathrm{FO}}\varphi$,
\item
$\{(\mathfrak{M},f)\}\models^-\varphi'$ iff
$(\mathfrak{M},f)\not\models_{\mathrm{FO}}\varphi$.
\end{itemize}
\end{proposition}
It would be interesting and 
relatively easy to extend in a natural way \emph{the first-order
part}\footnote{The part without operators $Cx$.} of the above framework to involve 
generalized quantifiers (following \cite{double}).
Another option would be to 
consider operators that give a Turing-complete 
formalism (following \cite{tc} or even \cite{rub}\footnote{\emph{One} of the
main reasons for defining the 
Turing complete logic $\mathcal{L}$ in \cite{tc} is to enable the study of standard
logic problems in that framework. Indeed, studying complexities of finite 
satisfiability and finite validity problems makes a lot of 
sense in the framework of $\mathcal{L}$, while first-order logic FO is not the
right framework for related studies; studying fragments of $\mathcal{L}$ 
makes more sense than studying only fragments of FO. First-order logic is
too weak to capture standard computational logics with recursion mechanisms, while $\mathcal{L}$ contains
such logics as direct fragments almost. Also, doing descriptive 
complexity in the framework of $\mathcal{L}$ makes sense, as $\mathcal{L}$
captures RE and thus all standard complexity classes correspond to fragments of
the umbrella logic $\mathcal{L}$. Using an umbrella logic enables one to directly identify 
how logical constructors lead to increases in expressivity. Also $\mathcal{L}_{RE}$ of \cite{rub}
can be used as a basis for different kinds of studies in a similar way as $\mathcal{L}$ and
due to analogous reasons.

Capturing complexity classes with $\mathcal{L}$ and $\mathcal{L}_{RE}$ is
quite easy. One can for example modify the semantics of $\mathcal{L}$ by putting
limits to the number of times a loop can be entered. 
For example, one can dictate that each loop can 
be used only a polynomial (in the model domain 
size) number of times, or exponential, or whatever. This is involves using
clocked loops in the same way as in 
\cite{mucalc}, but with more adjusted limits.
 One can also limit the amount by how many points domains can be expanded in
both logics $\mathcal{L}$ and $\mathcal{L}_{RE}$. Capturing, e.g., ExpSpace and PSpace is
rather easy. And even many of the higher
classes are also easy to capture nicely due to the domain expansion capacity.

Lack of recursion is not the only weakness of FO. It is quite striking (and says something
about the state of logic as a field) that none of the paradigmantic computational logics has, e.g.,
majority quantifiers `for most $x$' in them. This is
striking because \emph{similarity} relations are \emph{central} for
real-life reasoning, and most similarity relations are defined in terms of majority statements and
the like. (Probabilistic logics are of course relatively widely studied, however, so things are not all that bad.)
The logic $\mathcal{L}$ can simulate typical majority statements and whatnot, so $\mathcal{L}$
banishes most related problems. The expressivity of $\mathcal{L}$ is, indeed, of a
\emph{fundamental} nature. The
same holds not for FO, despite Lindstr\"{o}m's theorems. Of course $\mathcal{L}$ is
fundamental mainly in relation to finite models, but still.

Classifying fragments of FO is, nevertheless, interesting and reasonably relevant. A proper 
classification scheme for fragments can be easily based on an algebraic approach. Indeed, 
cylindric algebras and Codd's theorem  offer obvious clues how we get access to fragments in a
proper way. While prefix classes of FO offer a nice starting point for classifications, it is a too crude
approach. Even the Guarded Fragment and $\mathrm{FO}^2$ fall outside prefix classes, not to 
mention logics with more intricate limitations, such as limited uses of $\neg$ and whatnot. A lot 
remains to be said about this issue, but this is getting too irrelevant to the main text.
Wittgenstein's Tractatus is nice, but similarly, probably some eighty percent of the text is irrelevant to
the main argument there.
}),
possibly following a direct game-theoretic approach
rather than the team semantics flavoured one given above. This would lead to
formalisms for parallelism and distributed computation
when used with model sets as opposed to models. 
However, while these generalizations can be done such
that the resulting formalisms are easily seen natural,
the formalism here that uses $Cx$ is 
harder to interpret especially if we allow for $Cx$ in 
the scope of negations.
If we use the semantics in formulae with $Cx$ occurrences,
then disjunction together with $Cx$ can become peculiar.\footnote{Note
that $\varphi\vee\psi$ simply means $\neg(\neg\varphi\wedge\neg\psi)$ here.}
Indeed, consided the model set $\mathcal{M}$ with two disjoint
models and nothing else. Now $\mathcal{M}$ satisfies\footnote{See \cite{det} for similar considerations.} $Cx(x=x)\vee C(x=x)$
while not satisfying $Cx(x=x)$. Thus the reading of $\vee$ indeed is
should be ``\emph{there are two cases such that $\varphi$ and $\psi$}" or
even ``\emph{the possibilities split into two cases
such that in the first case $\varphi$ and in the second case $\psi$}."
(Note that dependence logic
requires a similar reading of $\vee$ to be natural.)
Obviously, implications (e.g., $P(x)\rightarrow \varphi$
which stands for $\neg P(x)\vee \varphi$) can of course be
read ``\emph{if $P(x)$, then $\varphi$}" but
also ``\emph{in the \emph{case} $P(x)$, we have $\varphi$}."

These are natural readings especially if one is attempting to unify
semantics and proofs,
thereby relating $\vee$ with the proof by cases protocol.
Adopting the perspective that a model set is (intuitively) a
single fixed but unknown object (for example any
group from a collection of groups that extend a particular single group\footnote{Groups 
have a relational representation here since we
are considering relational signatures.}) is 
very natural and in such a framework it is natural to
make statements about splitting into cases.  ("The 
(unknown) group $G$ has property $P$
or $G$ has property $Q$...")\footnote{Similar 
statements are omnipresent. Further operators arise for related
statements, such as "It is possible that $G$ has property $P$," etcetera.
This modality statement obviously seems to say (more or less) that
the subset (of the current model set) where $G$
has $P$ is nonempty.} This is true 
especially because proofs are often (or 
almost always) made for a \emph{fixed
but unknown object} or objects. Thus the above semantics works for 
formalising that kind of thinking. Category theory of course can also be thought to operate
this way but here we have a very simple logic that can
also directly speak about the internal structure of objects.
It is obviously easy to expand the above framework, but we 
shall leave that for later.
\section{Satisfiability and applications}
We say that a sentence $\varphi\in L_C$ is \emph{satisfiable} if
there is some nonempty model set $\mathcal{M}\models^+\varphi$.
The satisfiability problem for a fragment $F$ of $L_C$ takes a sentence of $F$ as an input and asks
whether some nonempty model set satisfies $\varphi$, i.e., whether $\mathcal{M}\models^+\varphi$
for some nonemtpy model set $\mathcal{M}$.

The two-variable fragment of $L_C$ is the set of formulae 
that use instances of only the two variables $x$ and $y$.
We next show a complexity result concerning the two-variable fragment of $L_C$,
although it is easy to see that the related argument rather flexibly 
generalizes to suitable other fragments as well.
We discuss two-variable logic for convenience and also as it
and its variants (even in the team semantics context) have
received a lot of attention in recent years, see, e.g.,
\cite{worrell, boja, onedimensional, onedimsecond, infcomp, double, zeume}.
\begin{proposition}
The satisfiability problem of the two-variable
fragment of $L_C$ is $\mathrm{NEXPTIME}$-complete.
\end{proposition}
\begin{proof}
Define the following translation $T$ from $L_C$ into $\mathrm{FO}$,
where  $D$ is a fresh unary relation symbol (intuitively representing
the common domain of a model set).
\[
\begin{array}{ll}
T(Rx_1...x_k) &=\ \ Rx_1...x_k\\
T(x=y) &=\ \ x=y\\
T(\neg\varphi) &=\ \ \neg T(\varphi)\\
T((\varphi\wedge\psi)) &=\ \ (T(\varphi)\wedge T(\psi))\\
T(\exists x \varphi) &=\ \ \exists x T(\varphi)\\
T(C x \varphi) &=\ \ \exists x (D x\wedge T(\varphi))\\
\end{array}
\]
We will prove below (in a couple of steps) that a
formula $\varphi$ of $L_C$ is satisfied by
some nonempty model set iff $T(\varphi)$ is satisfied by some model (in
the classical sense).
This will conclude the proof of the current proposition as it is
well known that the satisfiability problem of two-variable 
first order logic is $\mathrm{NEXPTIME}$-complete.
We first note that if $T(\varphi)$ is satisfiable by some
model $(\mathfrak{M},f)$, then we have $\{(\mathfrak{M},f)\}\models^+T(\varphi)$ by
Corollary \ref{cor}. From here it is very easy to show that $\varphi$ is
satisfiable by the same model set by evaluating formulae step by step from outside in
(and recalling the syntactic restrictions of $L_C$.)
Thus it now suffices to show
that if some nonempty model set satisfies $\varphi$,
then $T(\varphi)$ is satisfied by some model.
To prove this, we begin by
making the following auxiliary definition.
Let $\mathcal{M}$ be a model set
and let $\mathcal{M}_D$ denote the model set 
obtained from $\mathcal{M}$ by adding a unary
predicate $D$ to each model that covers exactly the common domain of $\mathcal{M}$.
Recall that we have already fixed $\varphi$.
%
%If $(\mathfrak{M},f)\in\mathcal{M}$, we let ${\mathfrak{M}_D}$
%
%denote the model obtained by expanding the
%
%vocabulary of $\mathfrak{M}$ by $\hat{C}$
%
%such that $\hat{C}^{\hat{\mathfrak{M}}}$ is the
%
%common domain of $\mathcal{M}$.
%
%We let $\hat{\mathcal{M}}$
%
%denote $\{(\hat{\mathfrak{M}},f)\, |\, (\mathfrak{M},f)\in\mathcal{M}\, \}$.
%
%
%

\medskip

\noindent
\textbf{Claim.} $\mathcal{M}\models^+\varphi$ implies $\mathcal{M}_D\models^+T(\varphi)$.

\medskip

\noindent
The claim is easy to prove by evaluating formulae from
outside in using the semantics for model sets.
Assume that $\mathcal{M}\models^+\varphi$ for some nonempty
model set $\mathcal{M}$. Thus $\mathcal{M}_D\models^+ T(\varphi)$ by
the claim. Thus $(\mathfrak{M},f)\models_{
\mathrm{FO}} T(\varphi)$ for all $(\mathfrak{M},f)\in\mathcal{M}_D$ by Proposition \ref{fok}.
Therefore (since $\mathfrak{M}_D$ is nonempty) we have $(\mathfrak{M},f)
\models_{\mathrm{FO}} T(\varphi)$ for some $(\mathfrak{M},f)\in\mathcal{M}_D$.
This concludes the very easy proof.
\end{proof}

Going from perfect to imperfect information is in
general extremely easy to justify in several 
ways, so let us look at more concrete and even
rather specific and particular possible applications of model sets.

Ontology-based data access and related querying 
frameworks obviously offer a natural application for
model sets. The work there is rather active, see, e.g., \cite{obda, icdt17}
and the references therein. Another obvious and quite 
different application is distributed computing. 
One (of many) ways to model a computer network via logic would be to
combine the approaches of \cite{kucsl13} (which
accounts for communication) with \cite{tc} (which 
accounts for the local (Turing-complete) computation).
The nodes of \cite{kucsl13} would become first-order models, so
the domains considered would be model sets (with relations that
connect models to other models). See also \cite{kuu14gan} for
some (simple) adaptations of the framework in \cite{kucsl13}.

For yet another example, let $\varphi$ and $\psi$ denote your favourite
theorems. One can now ask:
``Does $\psi$ \emph{follow from} $\varphi$?''
The first answer could be: ``Yes, since $\varphi$ is a
true theorem, it in fact already follows from an empty set of
assumptions.'' The next answer could be a bit more 
interesting: for example, if $\varphi$ and $\psi$ were
theorems of arithmetic, one could try to
investigate if $\psi$ follows from $\varphi$ as a logical
consequence, i.e., even without the axioms of arithmetic.
Different approaches to \emph{relevance} have been
widely studied, and the example below is not unrelated to that. 
Let $D$ be a deduction system (or some conceptually similar algorithm). Now, for each $n\in\mathbb{N}$, let $\rightarrow_n^D$ denote the connective defined such that $\varphi\rightarrow_n^D\psi$ holds if $\psi$ can be deduced from the premiss $\varphi$ in $n$ deduction steps (applications of deduction rules) in $D$. Here `$\varphi\rightarrow_n^D\psi$ holds' is a metalogical statement; we could consider closing the underlying logic under $\rightarrow_n^D$ and the other connectives, but we shall not do that now. We note that also statements $(\varphi_1,...,\varphi_k)\rightarrow_n^D\psi$, containing several premises, can be introduced. 
Statements $\varphi\rightarrow_n^D\psi$ capture aspects of relevance. The idea here is that whether $\psi$ follows
from $\varphi$, depends on the particular background
knowledge and abilities as well as the computational capacity (of an agent, for example). With this interpretation, it is indeed highly \emph{contingent} whether something follows in $n$ steps from something else. It here depends on the particularities of $D$. Also, how immediately something follows from something else, is a matter of degree. This is the role of the subindex $n$.

This kind of a framework is one example (of many) that can be elaborated in a possibly more interesting way by using an approach to proofs that is directly (indeed, directly) linked to semantics, with connectives corresponding to proof steps. Model sets offer such possibilities in a natural way. It is worth noting that also refutation calculi (rather than proof calculi), and generalizations thereof, fit into the framework well. The related approaches can be based on the dual
systems of \cite{double}. That framework obviously offers quite natural possibilities for generalizations of model sets as well.

Another natural and related approach is to consider extensions of ATL\footnote{ATL stands for alternating time temporal logic.} with individual states replaced by relational structures. The players then modify the relational structure in every step, leading essentially to a computation tree with nodes corresponding to relational structures.\footnote{It can be quite natural to let strategies be determined by the relational structure only, which means that positional strategies are used. Trivially, incomplete information can be modeled---with some success---with model sets (instead of individual structures) as states. (Of course these can simply be considered \emph{sets} of states, if desired.)}
This is a very general approach. In the particular case of modeling proofs or evolving information states, the relational structures can simply be (or encode) sets of formulae for example. Note that even paraconsistent states are quite easy to handle here, as they are simply inconsistent sets of formulae.
More on this in section \ref{generalsystems} below.
It is especially interesting to consider systems where the individual players take actions,
and those actions plus the action of nature then
computably (or in a semi-decidable way) produce the new relational model.
The moves are determined by the previous relational structure and
the actions. Even distributed computing
systems and beyond are natural in the framework. Note that even infinite structures make
immediate sense here if there exists a \emph{perception function} that returns a 
finite structure from each infinite structure, and that finite structure is then used by at 
least the agents' strategies. Note that different agents' strategies can even 
depend on different (but probably overlapping) relation symbol sets. In any case, it is
natural to make the strategies to depend solely on the current relational structure; any memory
ought to be encoded in that structure (and can be visible to only a single agent since 
the agents can see different relation symbol sets). This is a nice way to model the interaction of
minds together with the material world.
Multiperspective thought provides another
example of immediate applications of model sets. Such thought 
seems to be considered controversial by 
many. Yet, it is mostly very simple, and it is indeed 
surprizing that it is so often 
considered problematic.
%Much of
%the issue is simply false reading of
%statements of the standard form $Px$ ($x$ is $P$, or $x$
%has the property $P$).
%
The difficulties in understanding related 
statements are often due to the assumption of
bivalence and the assumption that 
concepts have fully fixed meanings in
contexts where such assumptions are na\"{i}ve.\footnote{Indeed,
many---if not most---philosophical problems stem from
underdeterminacy of concepts. Here `underdetermined' could mean
`not fully defined' and determined  `fully defined.' Alternatively,
`underdetermined' could here stand for 
`not specified up to a sufficient extent' and determined
for `specified up to a sufficient extent.' Philosophical confusion quite typically arises from considering 
underdetermined concepts determined.
A related and highly relevant demarcation 
problem is to \emph{try} to determine which questions can be
naturally turned into `determined' questions and which not. Here a prima facie idea would be that in
favourable cases, a question turns into a set of determined questions, one
for each natural interpretation, with each
determined question being associated with definitions that force determinacy.
The issue is then to consider how natural and 
appropriate those differing sets of definitions are (and also to
solve the---unambiguous but open---determined questions). In less favourable
cases, the question simply escapes all attempts to banish ambiguities, due to intrinsic 
ambiguities and
finite resources for the classification process. A sometimes sufficient `metasolution'
here could perhaps be that no solution can be obtained.}
Let us consider a very simple formal framework that captures---and
thus elucidates---at 
least some aspects of multiperspective thought.
%The relations to model sets are obvious but not elaborated below.

%
Consider a model set $\mathcal{U}$ which we shall call the \emph{universe}.
A \emph{property} is a subset $\mathcal{P}\subseteq\mathcal{U}$. (Properties
need not be closed under isomorphism.) A \emph{weight function} is a
mapping $w:\mathit{Pow}(\mathcal{U})\rightarrow S$, where $S$ is some
non-empty set and $\mathit{Pow}$
the power set operator; the set $S$ could be, for example, the set of
real numbers $\mathbb{R}$. We call the set $S$ the \emph{set of weights}
and the elements $s\in S$ are obviously called \emph{weights}.
Let $S_m$ denote the set of multisets over $S$, i.e., collections of
elements of $S$ that also enable different multiplicities of elements to occur.\footnote{We
allow for infinite multipilicities, but limit the largest possible multiplicity with
some sufficiently large cardinal, for example something greater than 
the power set of $\mathcal{U}$.}
A function $E:S_m\rightarrow V$ is called an \emph{evaluation function},
where the set $V$ is an arbitrary set of \emph{values} $v\in V$. For example,
for finite $\mathcal{U}$ and with $S = \mathbb{R}$, the
function $E$ could be the operator that gives the
sum of any collection of inputs.
Now, let $s$ be a one-to-one function
from $\mathit{Pow}(\mathcal{U})$ into a set of \emph{statements}, so
each property $\mathcal{P}$ is associated with a statement $s(\mathcal{P})$.
The weight of the statement $s(\mathcal{P})$ is $w(\mathcal{P})$
and the value of a set $K$ of statements is $E(\{ w(\mathcal{P})\ |\ s(\mathcal{P})\in K\})$,
where the argument set is a multiset of weights of statements.
So, we (or a group of people) can know that different properties hold, i.e., we know the 
actual model is inside different sets $\mathcal{P}\subseteq\mathcal{U}$. Each 
property $\mathcal{P}$ contributes a weight. We can possibly combine the 
weights and get different values, depending on which properties
are involved. For example, some true 
properties can contribute a negative number (as a weight) and others a
positive one. The \emph{full value} is the value obtained by
considering the multiset of weights of all properties. (For example, the
full value could be the sum of all weights.)
Now, the full value then, in the end, can be
associated with a truth value, if desired. Note that it often natural to take 
the intersection of all known properties first and then associate that one 
property with a weight.

For example, I can
state that John is rich as he has a million dollars on his
bank account, and Jill can state that John is not rich as he has a
debt of two million dollars. Obviously these observations of 
partial knowledge (i.e., John has money and debt) are simply contributions towards an ultimate 
picture, and the possible seeming syntactic contrariness of the related partial statements
(John is rich and John is not rich) amounts to
nothing much at all.
The next section shortly discusses a fresh link between logic and combinatorics.
The connection to model sets is kind of trivial but nice. Every combinatorial
property $P$ and input structure domain size $n$ will give rise to a model set $P_n$ which
contains the models with property  $P$ and the domain size $n$. The domain of 
the models can naturally be considered to be $\{0,\dots , n-1\}$. The size of $P_n$ is
related to the counting enumeration function for $P_n$.

\subsection{The interplay of logic and combinatorics}

What is enumerative combinatorics? This is, of course, a 
philosophical question. The business of enumerative combinatorics
often follows the following pattern:

\begin{enumerate}
\item
Input: a property.
\item
Output: the enumeration function for the property.
\end{enumerate}

\noindent
Perhaps a more accurate picture would be the following.

\begin{enumerate}
\item
Input: a property.
\item
Output: a formula for the enumeration function for the property.
\end{enumerate}

For example, we could be asked how many symmetric binary
relations there are on an $n$-element set.
There are $2^{\binom{n}{2} + n}$ symmetric binary relations on an $n$-element
set, so here the input would be the property of being a symmetric binary relation, 
and the output would be the formula $2^{\binom{n}{2} + n}$. Another example
could be the following.

\begin{enumerate}
\item
Input: being an anti-involutive\footnote{Here we define a 
function to be anti-involutive if $f(f(x))\not=x$ for all inputs $x$.} function $f:n\rightarrow n$.
\item
Output: a formula for the enumeration function for the property.
\end{enumerate}

\noindent
The enumeration function is given by
$${\sum\limits_{i = 0}^{i = \lfloor n/2\rfloor}(-1)^i(n-1)^{n-2i}
\binom{n}{2i}\frac{(2i)!}{2^i\cdot i!}}$$
which follows, for example, as a special case of Proposition 3.1 of \cite{kuusilutz}.
Now, an interesting and relevant idea would be to formalize both the input and 
output in our general scheme for enumerative combinatorics.
There is a lot of freedom in the way this can be done. For example, the input 
properties could be formalized in first-order logic or even (fragments of) some stronger logic
such as the Turing-complete logic from \cite{tc}.\footnote{Also the Turing-complete
logic $\mathcal{L}_{\mathrm{RE}}$ from \cite{rub} would be interesting here.} On the
output side, also many
different formalisms could be interesting. Arithmetic formulae could be
constructed from $+,\cdot,\sum,\prod$ etcetera, starting from a variable $n$ (or even 
several variables). Different kinds of results would be obtained for different sets of 
functions and operators. Subtraction, division, exponentiation, factorials, and 
operators for all kinds of basic operations\footnote{In addition to $\sum,\prod$, a whole
range of operations comes to mind, even ones for constructing recurrence relations etcetera.
The sky is the limit.}
could be used in various different combinations. Our above example would become
formulated as follows.

\begin{enumerate}
\item
Input: ${\forall x \forall y
\neg(Rxy\wedge Ryx) \wedge\forall x\exists^{=1}y\, Rxy}$.
\item
Output: ${\sum\limits_{i = 0}^{i = \lfloor n/2\rfloor}(-1)^i(n-1)^{n-2i}
\binom{n}{2i}\frac{(2i)!}{2^i\, \cdot\, i!}}$.
\end{enumerate}

There exists no theory based on this exact idea, so it must be built. There is indeed a
lot of freedom here. 
The idea of the suggested research programme is to
understand the interplay of logical operators and 
arithmetic functions (and operators). How do arithmetic expressions arise from the 
logic-based specifications combinatorial properties, and vice versa? What can we learn 
from those connections?

However, something is known of course. For example, already \cite{kuusilutz} gives 
the answer to the above input ${\forall x \forall y
\neg(Rxy\wedge Ryx) \wedge\forall x\exists^{=1}y\, Rxy}$ in an algorithmic way.
Indeed, from \cite{kuusilutz}, an algorithm can be extracted that takes as inputs formulae of 
two-variable logic with a functionality axiom and outputs formulae for the related
enumeration functions. If we care only about the complexity of the enumeration
function instead of arithmetic formulae describing it, then of course results in 
model counting are immediately interesting as well. For example, all properties expressible in 
two-variable logic have PTIME-computable 
enumeration functions even when a 
functionality axiom is used \cite{kuusilutz}.\footnote{The input here is
given in unary, conceptually being associated
with a structure domain rather than simply domain size.}

\subsection{Some general notions of a system}\label{generalsystems}

In this section we define some general notions of an \emph{evolving system}. No 
restrictions based on computability will be imposed at first. We begin
with formal definitions and give concrete examples after that.

Let $\sigma$ denote a (possibly infinite) vocabulary.
Let $A$ be an arbitrary set and $I$ an ordered set. Intuitively, $A$ is a set of
\emph{actions} and $I$ a set of indices or agent names.
%
%
%
\begin{comment}
We shall define three different notions of a \emph{system} and then
discuss the merits of each definition as well as the notion of a system in general.
When necessary, systems according to Definition \ref{firstsystem} (respectively,
\ref{secondsystem}, \ref{thirdsystem}) can be referred to as \emph{systems of
type} I (respectively, II, III).
\end{comment}
%

%
%
%
\begin{definition}\label{firstsystem}
\normalfont
A \emph{system frame base} (or simply a base) over $(\sigma,A,I)$ is a structure $(\mathcal{S},F)$ 
defined as follows.
\begin{enumerate}
\item
$\mathcal{S}$ is a set of $\sigma$-models.
The set $\mathcal{S}$ and any of the models can be infinite.\footnote{If desired, the
set $\mathcal{S}$ can even be a set of 
pairs $(\mathfrak{M},f)$, where $\mathfrak{M}$ is a $\sigma$-model and $f$ an
assignment mapping variables to the domain of $\mathfrak{M}$.}
%
%
%
%$f_i:\ \mathcal{S}
%\cup \{\mathbf{u}\} \ \ \rightarrow\ \ \ A\cup \{\mathbf{u}\}$.
%
\item
$F$ is a function $F:\ \mathcal{S}\times A^I\ \ \rightarrow\ \ \ 
\mathcal{P}(\mathcal{S})$. We require that $F$ returns a non-empty set. (This is 
neither crucial nor elegant but is \emph{here} done to simplify things. An empty output can 
be modelled by special features on output models.)
\end{enumerate}
A \emph{system frame} over $(\sigma,A,I)$ is structure $(\mathcal{S},F,G)$ 
defined as follows.
\begin{enumerate}
\item
$(\mathcal{S},F)$ is a system frame base as defined above.
\item
Let $T$ be the set of sequences $\bigl((\mathfrak{M}_i,\mathbf{a}_i)\bigr)_{0\leq i\leq k}$
such that the following conditions hold.
\begin{enumerate}
\item
$k\in\mathbb{N}$.
\item
$\mathfrak{M}_i\in \mathcal{S}$ and $\mathbf{a}_i\in A^I$ for each $i$.
\item
$\mathfrak{M}_{i+1}\in F(\mathfrak{M}_i,\mathbf{a}_i )$ for all $i\in\{0,\dots , k-1\}$.
\end{enumerate}
The set $T$ is called the set of \emph{finite proper evolutions of $(\mathcal{S},F)$.}
We define $G$ to be a function $G:T\ \ \rightarrow\ \ \mathcal{S}$ such that $G(t)
\in F(\mathfrak{M}_k,\mathbf{a}_k)$ for all $t\in T$, where $(\mathfrak{M}_k,\mathbf{a}_k)$ is 
the last member of the sequence $t\in T$.
%Occasionally we let finite 
%proper evaluations contain also an outcome model, so they end in a model.
%This will be clear from the context.
\end{enumerate}
A \emph{system} over $(\sigma,A,I)$ is a structure $(\mathcal{S},F,G,(f_i)_{i\in I})$ 
defined as follows.
\begin{enumerate}
\item
$(\mathcal{S},F,G)$ is a system frame as defined above.
\item
Each $f_i$ is a function $f_i:\ \mathcal{S}
\ \ \rightarrow\ \ \ A$.
\end{enumerate}
The set $\mathcal{S}$ is the \emph{domain} of $(\mathcal{S},F)$, $(\mathcal{S},F,G)$
and $(\mathcal{S},F,G,(f_i)_{i\in I})$. \ \ \ \qedsymbol
\end{definition}
%

\begin{comment}
%
\begin{definition}\label{secondsystem}
\normalfont
A \emph{system} over $(\sigma,A)$ is a structure $(\mathcal{S},(f_i)_{i\in I},F)$ 
defined as follows.
%
%
%
\begin{enumerate}
%
\item
$\mathcal{S}$ is a set of $\sigma$-models.
The set $\mathcal{S}$ and any of the models can be infinite.
%
\item
Each $f_i$ is a function $f_i:\ \mathcal{S}
\ \ \rightarrow\ \ \ A$.
%$f_i:\ \mathcal{S}
%\cup \{\mathbf{u}\} \ \ \rightarrow\ \ \ A\cup \{\mathbf{u}\}$.
%
\item
$F$ is a function $F:\ \mathcal{S}\times A^I\ \ \rightarrow\ \ \ 
\mathcal{P}(\mathcal{S})$.
%
\end{enumerate}
%
%
%
\end{definition}
%
\end{comment}
%

We often talk about \emph{frame bases} (or \emph{bases}) and \emph{frames} instead of
\emph{system frame bases} and \emph{system frames}.

Intuitively, the frame base of a system can be considered the 
\emph{material} or \emph{physical} part of the system, while $G$ and the functions $f_i$ are 
the \emph{non-physical} or \emph{non-material} part.
The functions $f_i$ can be associated with individual \emph{agents}\footnote{More 
accurately, the functions $f_i$ are agent behaviour strategies and indices in $I$
correspond to agents or agent place holders.},
while $G$ can be considered a higher force that
ultimately determines the final evolutive
behaviour of the system.\footnote{$G$ can be interpreted in 
many ways. It could simply be considered \emph{change} or 
\emph{luck}, to give one example. One of the main features of $G$ is that it removes
non-determinism from systems.} The agents choose
actions from $A$ based on the current model, and a
new model is then produced according to the function $G$ based on the chosen actions.
Let $M = (\mathcal{S},F,G,(f_i)_{i\in I})$ be a system over $(\sigma,A,I)$ with domain $\mathcal{S}$.
A pair $(M,\mathfrak{M})$, where $\mathfrak{M}\in\mathcal{S}$, is called an \emph{instance}.
We may also call $(M,\mathfrak{M})$ a \emph{pointed system}, in analogy with modal logic.
The set of finite \emph{evolutions} of the frame base $(\mathcal{S},F)$ is the set that
containins all finite proper evolutions of $(\mathcal{S},F)$ and 
all models $\mathfrak{M}\in\mathcal{S}$. The models can be thought of as
zero-step evolutions of $(\mathcal{S},F)$.
The set of finite evolutions of the  system $M$ is the 
set that contains all $\mathfrak{M}\in\mathcal{S}$ and all pairs $(t,\mathfrak{M}_{k+1})$
such that the following conditions hold.
\begin{enumerate}
\item
$t = (\mathfrak{M}_i,\mathbf{a}_i)_{i\leq k}$ for some $k\geq 0$, all 
models and action tuples being from $M$.
%
%\item
%$\mathfrak{M}_i\in \mathcal{S}$ and $\mathbf{a}_i\in A^I$ for each $i$.
%
%
%
\item
$\mathbf{a}_i = (f_i(\mathfrak{M}_i))_{i\in I}$ for each $i\leq k$.
\item
$\mathfrak{M}_{i+1} = G\bigl((\mathfrak{M}_j,\mathbf{a}_j)_{j\leq i}\bigr)$
for all $i\leq k$.
\end{enumerate}

An infinite evolution is defined similarly, in the obvious way, but
without the final model $\mathfrak{M}_{k+1}$. Note that the agents' choices are
determined by the current model rather than the 
sequence including also all the previous models and actions leading to
the current model. The interpretation of this is that histories are to be
encoded in the current model, i.e., the current material world. It is also
natural to consider memory, somehow encoded, rather than some full history.

While this framework is an intuitive way of thinking about systems,
there are other natural choices. We
shall look at an alternative way of defining systems such that
the non-physical part can (ultimately) be eliminated, more or less, from the picture.
This new way is, for most purposes, \emph{essentially} equivalent to the old definition above.

Consider a system frame $S := (\mathcal{S},F,G)$. We describe a
way to eliminate the function $G$.
For each $E = \mathfrak{M}_0,\mathbf{a}_0,\dots ,\mathfrak{M}_{k+1}$ that is a 
finite evolution of the frame base $(\mathcal{S},F)$,
we create a new model $\mathfrak{M}_E$ which has $\mathfrak{M}_{k+1}$ as a \emph{basic part}
and some encoding of the evolution sequence $E$ as a disjoint part, with elements labelled by 
some fresh predicate in order to be able to tell which part of the 
new model encodes $E$; also other fresh relation symbols can of course be used, for
example to encode the order in which the models occur, and the actions that lead from 
one model to another, etcetera. We then create a new
domain $\mathcal{S}'$ that includes $\mathcal{S}$ and all the new models, that is,
the new domain includes (encodings of) all finite evolution 
sequences $E$ of $(\mathcal{S},F)$.
The zero-step evolutions can be thought of as \emph{starting point models} in 
the new system frame base we are about to create. We define a new index set $I'$
by adding a new index $J$ to the beginning of $I$ in order to be able to
accommodate a new function $g_J$ that simulates functions $G$.\footnote{Formally, $I'$ is an
ordered set that begins with $J$ after which come the elements of $I$.}
We modify $F$ to a new 
function $F': I'\rightarrow A'$, where $A'$ extends $A$, in the way described next.
We \emph{first} define a new 
function $F_G'$ that informally speaking always outputs the set of models (now extended with the 
history) that the function $G$ would also have given. More formally,
let $\mathfrak{M}$ be a model in the new domain, and let $\mathbf{a}\in {A}^{I}$.
Let $E = ((\mathfrak{M}_i,\mathbf{a}_i)_{i\leq k},\mathfrak{N})$ denote 
the evolution (containing models from the old domain) that is
encoded in the structure of $\mathfrak{M}$.
Thus $\mathfrak{N}$ is the model in the old
domain that $\mathfrak{M}$ corresponds to, i.e., $\mathfrak{N}$ is the current model 
that $\mathfrak{M}$ encodes. We define $F_G'(\mathfrak{M},\mathbf{a})$ to be the
model whose current part is $\mathfrak{M}' = G(E\cdot (\mathfrak{N},\mathbf{a}))$ and which 
encodes the evolution $E\cdot \mathbf{a}\cdot \mathfrak{M}'$.
Notice that $F_G'$ is now deterministic. The frame base $(\mathcal{S}',F_G')$ 
simulates $(\mathcal{S},F,G)$ in the obvious way.
Note that $F_G'$ covers only one particular function $G$.
We can modify the framework by defining a deterministic function $F'$ which
takes into account different possible functions $G$. Indeed, we can define $F'$ so that the
novel index $J$ can accommodate different functions\footnote{Note that these
functions are simply strategies of an
agent.} $g_{J}$ that simulate 
functions $G$. We include in $A'\supseteq A$ a single choice for each possible
behaviour pattern that give an outcome model based on the 
moves of the old agents $f_i$ with $i\in I$. Then 
the function $F'$ can easily be modified to deterministically provide the
outcome that $G$ would force. We can make $A'$ small is 
all features in our framework are suitably regular.
It is possible to require that the functions $G$ and $f_i$  (and even $F$) of different
systems are somehow part of (or encoded in) the models. This is easy to accomplish by
suitable encodings. For example, we can easily encode Turing machines into the
models in system domains.
It is possible even to modify these machines on the fly and dictate that if no output is given, or if
there is even a syntax error in the encoding, then some default action is taken.
These kinds of systems can be called \emph{fully material}. The requirement is that $F$, $G$
and each $f_i$ behave as their material counterparts (encoded in the models) would 
dictate. Now systems become 
essentially equivalent to system domains, but of course fixed background 
definitions are needed if we somehow try to specify systems only by giving 
the domains (in one way or another), as otherwise we cannot necessarily
determine the intended evolutionary behaviour of the systems.
While the functions $f_i$ can depend on all of the current model, which is 
natural when modelling perfect information games, it is also natural to define
\emph{perception functions} and make functions $f_i$ depend upon \emph{perceived 
models}. We can let an individual perception function be a
map $\varphi:\mathcal{S} \rightarrow \mathcal{S}'$, where $\mathcal{S}'$ contains
models whose signature may be different from those in $\mathcal{S}$. Now, even if
the input model to $\varphi$ is infinite, the output model can be finite and depend only on
some small part of the input model. We can now dictate that $f_i(x) = f(\varphi(x))$
for each input $x$, where $f$ gives an action in $A$. For a concrete example, $\varphi$ 
could be a first-order reduction, more or less in the sense of descriptive complexity, that gives a 
very crude, finite approximation of the original model. Note that agents' 
epistemic states can be somehow part of the original models, so agents can 
take these into account up to one extent or another.
It is also natural and easy to tie the functions $f$ to \emph{material bodies}. These
can be modelled by, e.g., specially reserved predicate symbols or some more
general constructions. The perception functions can be relativised to depend on 
only the substructures that the bodies of agents (almost) cover. Note that there is no
problem in letting the bodies of different agents overlap. It is natural (but by no means
necessary) to let the encodings of the machines that govern functions $f_i$  to be
part of the bodies of the related agents.

While cellular automata are an ok starting point for digital physics, the above described
approach (and its numerous trivial variants) offer a much richer modelling framework for 
related formal approaches to physics.\footnote{Of course one of the first ideas is to keep 
more or less all functions computable or semi-computable. Partial 
functions can be suitably accommodated into the system of course.}
The metaphysical setting of such formal frameworks
offers a lot of explanatory power for the deeper level nature of physical (and other) phenomena.
The interplay of the supposedly mental constructs ($G$ and each $f_i$) with the material 
parts is obviously highly interesting. The fully formal nature of the systems will simply \emph{force} 
new concepts and insights to emerge as the result of concrete modelling attempts.
Importantly, it is not at all necessary to always keep everything computable or recursively 
enumerable, even though such limitations are an obviously 
interesting and important case. Nor is there any reason to force entities (agents) to be 
somehow \emph{local} in models.

As suggested in \cite{antti15kuu}, extensions of the Turing-complete logic can be
naturally used as logics to guide such systems (when using semi-computable functions).

\subsection{Hierarchical approaches}
Here we look at ways of dealing with modalities.
There are various \emph{prima facie reasonable} ways to 
proceed. Let us begin with one.
Let $\mathcal{M}$ be a set of pointed Kripke models. (By fixing some variable, say $x$, 
sets of Kripke models can be associated with model sets: a pointed model $(\mathfrak{M},w)$ is
identified with $(\mathfrak{M},f)$ where $f(x) = w$.)
Let $(P_\alpha)_{\alpha\in\mathbb{Z}_+}$ be a
sequence such that
\begin{enumerate}
\item
$P_1 \subseteq \mathcal{M}$.
\item
$P_{\alpha+1} \subseteq \mathcal{P}(P_{\alpha})$.
\footnote{Here $\mathcal{P}$ is 
the power set operator, so $\mathcal{P}(S_{\alpha})$ is the power set of $S_{\alpha}$.}
\end{enumerate}
We call $P_{\alpha}$ a
\emph{perspective} of rank $\alpha$. It makes sense to
not allow the empty set to belong to perspectives. We 
thus do so here.
Consider the language of modal logic,
$$\varphi\ ::=\ P\ |\ (\varphi\wedge\varphi)\ |\ \neg\varphi\ |\ \Diamond\varphi\ $$
where $P$ is any unary predicate (i.e., a proposition symbol) in the signature of
the models of $\mathcal{M}$.
Define the \emph{rank} $r(\varphi)$ of a formula $\varphi$ in 
the natural way as follows.
\begin{enumerate}
\item
$r(\chi) = 0$ for each atomic formula $\chi$.
\item
$r(\neg\chi) = r(\chi)$.
\item
$r(\chi\wedge\psi) = \mathit{max}\{r(\chi),r(\psi)\}$.
\item
$r(\Diamond\varphi) = r(\varphi) + 1$.
\end{enumerate}
A perspective $P_{\alpha}$ of rank $\alpha$ can interpret
formulae of rank $\alpha$ or less with
the following semantics.
%

%
\begin{comment}
Assume first that $\alpha = 0$. Then
%
%
%
\[
\begin{array}{ll}
%
%
%
P_{0}\models^+ P\
%
&\text{ iff }\ (\mathfrak{M},w)\models P \text{ holds for all }
%
(\mathfrak{M},w)\in P_0\\
%
P_{0}\models^- P\
%
&\text{ iff }\ (\mathfrak{M},w)\not\models P \text{ holds for all }
%
(\mathfrak{M},w)\in P_0\\
%
\end{array}
\]
%
%
%
where $\models$ is the standard semantic turnstile of Kripke semantics.
%
\end{comment}
%

%
Assume first that $r(\varphi) < \alpha$. Then 
\[
\begin{array}{ll}
P_{\alpha}\models \varphi\
&\text{ iff }\ P_{\alpha - 1}\models \varphi \text{ for all }
P_{\alpha - 1}\in P_{\alpha}\\
%
%P_{\alpha}\models^- \varphi\
%
%&\text{ iff }\ P_{\alpha - 1}\models^{-} \varphi \text{ for all }
%
%P_{\alpha - 1}\in P_{\alpha}\\
%
\end{array}
\]
Note that if $\alpha = 1$, then $P_{\alpha - 1}$ is a 
pointed model in $P_{\alpha}$. Then the above
holds with the natural definition that 
$P_{\alpha - 1} = (\mathfrak{M},w) \models \varphi$
iff we have $(\mathfrak{M},w)\models_K \varphi$ 
where $\models_K$ is the standard semantic
turnstile of Kripke semantics.
Assume then that $r(\neg \psi) = r(\chi\wedge \psi) = r(\Diamond \varphi) = \alpha$. Then 
\[
\begin{array}{ll}
P_{\alpha}\models (\chi \wedge \psi)
&\text{ iff }\ P_{\alpha}\models \chi
\text{ and }P_{\alpha}\models \psi\\
P_{\alpha}\models \neg\psi
&\text{ iff }\ P_{\alpha}\not\models \psi\\
P_{\alpha}\models \Diamond\varphi
&\text{ iff }\ P_{\alpha - 1}\models \varphi\text{ for some }
P_{\alpha - 1}\in P_{\alpha}
\\
\end{array}
\]

\subsubsection{First-order modal logic and beyond}\label{fossst}

We then consider first-order modal logic and beyond. The approach is very similar to
the one above, but we
spell it out anyway, because we can.
Let $\mathcal{M}$ be a model set, i.e., a set of pairs $(\mathfrak{M},f)$ 
where $f$ is some assignment for $\mathfrak{M}$, interpreting
first-order variables.
Let $(P_\alpha)_{\alpha\in\mathbb{Z}_+}$ be a
sequence such that
\begin{enumerate}
\item
$P_1 \subseteq \mathcal{M}$.
\item
$P_{\alpha+1} \subseteq \mathcal{P}(P_{\alpha})$,
where $\mathcal{P}$ is again
the power set operator.
\end{enumerate}
We call $P_{\alpha}$ a
\emph{perspective} of rank $\alpha$. Again it often makes sense to 
not allow the empty set to belong to or be a perspective. Also, it is 
relatively natural to require the \emph{domain} of all the models in the
model set $\mathcal{M}$ to be the same. A perspective $P_{\alpha}$ is
called \emph{regular} if it is built from a model set where all models 
have the same domain. It is \emph{strongly regular} if we also have $P_{\alpha'}\not=\emptyset$
for all perspectives $P_{\alpha'}$ that are part of $P_{\alpha}$ on the 
different rank levels of $P_{\alpha}$ (with
the perspective $P_{\alpha}$ itself also nonempty).
Below we study strongly regular perspectives only.
Consider the language of first-order modal logic,
$$\varphi\ ::=\ Rx_1\dots x_k\ |\ (\varphi\wedge\varphi)\ |\ \neg\varphi\ |\
\exists x \varphi\ |\ \Diamond\varphi\ $$
where $Rx_1\dots x_k$ is any atom.
Define the \emph{rank} $r(\varphi)$ of a formula $\varphi$ as follows.
\begin{enumerate}
\item
$r(\chi) = 0$ for each atomic formula $\chi$.
\item
$r(\neg\chi) = r(\chi)$.
\item
$r(\chi\wedge\psi) = \mathit{max}\{r(\chi),r(\psi)\}$.
\item
$r(\exists x \varphi) = r(\varphi)$.
\item
$r(\Diamond\varphi) = r(\varphi) + 1$.
\end{enumerate}
A perspective $P_{\alpha}$ of rank $\alpha$ can interpret
formulae of depth $\alpha$ or less with
the following semantics.
%

%
\begin{comment}
Assume first that $\alpha = 0$. Then
%
%
%
\[
\begin{array}{ll}
%
%
%
P_{0}\models^+ P\
%
&\text{ iff }\ (\mathfrak{M},w)\models P \text{ holds for all }
%
(\mathfrak{M},w)\in P_0\\
%
P_{0}\models^- P\
%
&\text{ iff }\ (\mathfrak{M},w)\not\models P \text{ holds for all }
%
(\mathfrak{M},w)\in P_0\\
%
\end{array}
\]
%
%
%
where $\models$ is the standard semantic turnstile of Kripke semantics.
%
\end{comment}
%

%
Assume first that $r(\varphi) < \alpha$. Then 
\[
\begin{array}{ll}
P_{\alpha}\models \varphi\
&\text{ iff }\ P_{\alpha - 1}\models \varphi \text{ for all }
P_{\alpha - 1}\in P_{\alpha}\\
\end{array}
\]
Note that if $\alpha = 1$, then $P_{\alpha - 1}$ is a 
model $(\mathfrak{M},f)$ in $P_{\alpha}$. Then the above
holds with the natural additional definition that 
$(\mathfrak{M},f) \models \varphi$
iff $(\mathfrak{M},f)\models_{\mathrm{FO}} \varphi$
where $\models_{\mathrm{FO}}$ is the
standard semantic turnstile of $\mathrm{FO}$.
Assume then that $r(\neg\psi)
= r(\chi\wedge \psi) = r(\Diamond\varphi) = \alpha\geq 1$. Then 
\[
\begin{array}{ll}
P_{\alpha}\models (\chi \wedge \psi)
&\text{ iff }\ P_{\alpha}\models \chi
\text{ and }P_{\alpha}\models \psi\\
P_{\alpha}\models \neg\psi
&\text{ iff }\ P_{\alpha}\not\models \psi\\
P_{\alpha}\models \Diamond\varphi
&\text{ iff }\ P_{\alpha - 1}\models\varphi\text{ for some } P_{\alpha - 1}\in P_{\alpha}\\
\end{array}
\]
%
%
%

\begin{comment}
An interesting alternative is to put 
%
$P_{\alpha}\models^+ (\chi \wedge \psi)
%
\text{ iff }\ Q\models^+ \chi
%
\text{ and }Q'\models^+ \psi\text{ for all }Q,Q'\subseteq P_{\alpha}
%
\text{ such that } Q\cup Q' = P_{\alpha}$. It is strange for 
diamonds though.
%
\end{comment}
%
%
%

\begin{comment}
To define a semantics for $\exists x$, we first give some auxiliary definitions.
The \emph{common domain} of an individual model $(\mathfrak{M},f)$ is 
simply the domain of $\mathfrak{M}$. The \emph{common domain} of a
perspective $P_{\alpha}$ is the intersection of the common domains of
all $P \in P_{\alpha}$. Here we assume $P_{\alpha}\not=\emptyset$.
\end{comment}
If $P$ is a perspective,
we let $P[a / x]$ denote the perspective where each model $(\mathfrak{M},f)$
existing on the rank $0$ level of $P$ is replaced by $(\mathfrak{M},f[a/x])$.
By the \emph{model domain} of $P$ we
refer to the domain of the models the (regular) perpective $P$ is built with.
%We define $(\mathfrak{M},f)[a/x] = (\mathfrak{M},f[a/x])$.
%
%
%
\begin{comment}
First, a \emph{choice function} of rank $\alpha > 1$ for a
perspective $\mathcal{P}$ of rank $\alpha$ is a
function $F: \mathcal{P}\rightarrow \mathcal{F}$, where $\mathcal{F}$ is
the set of choice functions of rank $\alpha - 1$. 
%
A choice function of rank $1$ is simply a choice 
function as defined in the previous sections. Thus a choice function of rank $1$
has model sets as domains. A choice function $F$ of rank $\alpha > 1$ has a
perspective $\mathcal{P}$ of rank $\alpha$ as a domain and it outputs
elements in the domains of the models that can be found from level $0$ of
the input perspective.
%

%
If $F$ is a choice function of rank $\alpha > 1$ and with domain $\mathcal{P}$,
we let $\mathcal{P}[F/x]$ denote
%
the perspective obtained from $\mathcal{P}$ by 
replacing each $\mathcal{P'}\in\mathcal{P}$ by $\mathcal{P}'[F'/x]$, 
where $F'$ is a choice function with domain $\mathcal{P}'$ mapping 
each input to $F(\mathcal{P}')$.
This ultimately gets down to the definition given for model sets. To be
properly defined, the
\end{comment}
%
%
%
Assume $\exists x\varphi$ and $P$ have the same rank.
We define
\begin{multline*}P\models \exists x \varphi\text{ iff }\
{P}[a/x]
\models\varphi\
\text{for some $a$ in the model domain of $P$.}
\end{multline*}
%
%
%
%
%
%
%
\begin{comment}
\begin{multline*}P\models^- \exists x \varphi\text{ iff }
\{{P'}[F({P}')/x]\ |\   {P'}\in {P}\\ \text{ and }F
\text{ is some function with}
\text{ domain }{P} \text{ and which
maps}\\ \ \ \ \ \ \ \ \ \ \ \ \ \ \ \ \text{each }{P'\in P}\text{ to some element in the common }\\
\text{domain of }P'
\}\models^-\varphi.
\end{multline*}
\end{comment}
%
%
%

\subsubsection{Propositional modal logic, a new take}\label{prop2}

Here we look at further ways of dealing with modalities.
Perspectives are as above for propositional modal logic.
(Thus we can flexibly talk about both pairs $(\mathfrak{M},f)$ and
pairs $(\mathfrak{M},w)$ in the discussions.)

Consider the following language of modal logic:
$$\varphi\ ::=\ P\ |\ (\varphi\wedge\varphi)\ |\ 
\ (\varphi\vee\varphi)\ |\ \neg\varphi\ |\ \Diamond\varphi\ $$
where $P$ is any unary predicate (i.e., a proposition symbol) in the signature of
the models of $\mathcal{M}$. This time we are \emph{not} going to
consider $\vee$ to be defined in terms of $\wedge$ in the usual way.
Define the rank of formulae as before, with $r(\chi\vee\psi)$
being defined the same as $r(\chi\wedge\psi)$.
A perspective $P_{\alpha}$ of rank $\alpha$ can again interpret
formulae of depth $\alpha$ or less with
the following semantics.
%

%
\begin{comment}
Assume first that $\alpha = 0$. Then
%
%
%
\[
\begin{array}{ll}
%
%
%
P_{0}\models^+ P\
%
&\text{ iff }\ (\mathfrak{M},w)\models P \text{ holds for all }
%
(\mathfrak{M},w)\in P_0\\
%
P_{0}\models^- P\
%
&\text{ iff }\ (\mathfrak{M},w)\not\models P \text{ holds for all }
%
(\mathfrak{M},w)\in P_0\\
%
\end{array}
\]
%
%
%
where $\models$ is the standard semantic turnstile of Kripke semantics.
%
\end{comment}
%

%
Assume first that $r(\varphi) < \alpha$. Then 
\[
\begin{array}{ll}
P_{\alpha}\models^+ \varphi\
&\text{ iff }\ P_{\alpha - 1}\models^+ \varphi \text{ for all }
P_{\alpha - 1}\in P_{\alpha}\\
P_{\alpha}\models^- \varphi\
&\text{ iff }\ P_{\alpha - 1}\models^{-} \varphi \text{ for all }
P_{\alpha - 1}\in P_{\alpha}\\
\end{array}
\]
Note that if $\alpha = 1$, then $P_{\alpha - 1}$ is a 
pointed model in $P_{\alpha}$. Then the above
holds with the natural definition that 

\[
\begin{array}{ll}
(\mathfrak{M},w) \models^+ \varphi\
&\text{ iff }\ (\mathfrak{M},w)\models \varphi\\
(\mathfrak{M},w) \models^- \varphi\
&\text{ iff }\ (\mathfrak{M},w)\not\models \varphi\\
\end{array}
\]
where $\models$ is the basic turnstile of Kripke semantics.
We then make the following auxiliary definition. Assume $\chi$ is of
some rank $\alpha_{\chi} < \alpha$. Consider a perspective $P_{\alpha}$.
We define $P_{\alpha}\upharpoonright\chi$ to be
the perspective (or rank $\alpha$) that 
can be obtained from $P_{\alpha}$ 
by removing each perspective $Q_{\alpha_{\chi}}$ from the rank $\alpha_{\chi}$
level such that $Q_{\alpha_{\chi}}\not\models^+\chi$. More formally, we
define $P_{\alpha}\upharpoonright\chi$ as follows.
Recall that $P_{\alpha}$ is defined inductively with the condition that
\begin{enumerate}
\item
$P_1 \subseteq \mathcal{M}$,
\item
$P_{\beta+1}\subseteq \mathcal{P}(P_{\beta})$.
\end{enumerate}
Define a new sequence $(Q_{\beta})_{1\leq\beta\leq\alpha}$
\begin{enumerate}
\item
If $\alpha_{\chi} = 0$, then $Q_1 = \{\, (\mathfrak{M},f)\in P_1\, |\,
(\mathfrak{M},f)\models \chi\ \}$. Otherwise $Q_1 = P_1$.
\item
Suppose we have 
defined $Q_{\beta}$.
\begin{enumerate}
\item
If $\alpha_{\chi} > \beta$,
then $Q_{\beta + 1} = P_{\beta + 1}$.
\item
If $\alpha_{\chi} = \beta$,
then $Q_{\beta + 1} = \{\, R_{\beta}\in P_{\beta + 1}\, |\,
R_{\beta}\models^+ \chi\ \}$.
\item
If $\alpha_{\chi} < \beta$,
then $Q_{\beta + 1} = \{\, R_{\beta}\upharpoonright\chi\, |\,
R_{\beta}\in P_{\beta+1},\ R_{\beta}\upharpoonright\chi\not=\emptyset\, \}$.
\end{enumerate}
\end{enumerate}
 We define $P_{\alpha}\upharpoonright\chi = Q_{\alpha}$.
We also define $P_{\alpha}\upharpoonright\overline{\chi}$ to
be equal to $P_{\alpha}\upharpoonright {\sim\chi}$ where $\sim$
denotes classical negation, i.e., if $\chi$ is of rank $0$,
then $(\mathfrak{M},f)\models {\sim\chi}$ iff
$(\mathfrak{M},f)\not\models \chi$, and if $\chi$ of rank $\beta$,
then $P_{\beta}\models^+ \sim\chi$ iff $P_{\beta}\not\models^+ \chi$.
Assume then that $r(\Diamond\varphi) = r(\psi)
= \alpha$ and $r(\chi) < \alpha$. Then 
\[
\begin{array}{ll}
P_{\alpha}\models^+ (\chi \wedge \psi)
&\text{ iff }\ P_{\alpha}\models^+ \chi
\text{ and }P_{\alpha}\models^+ \psi\\
P_{\alpha}\models^+ (\psi\wedge\chi)
&\text{ iff }\ P_{\alpha}\models^+ \psi
\text{ and }P_{\alpha}\models^+ \chi\\
P_{\alpha}\models^+ (\chi \vee \psi)
&\text{ iff }\
\bigl(\ (P_{\alpha}\upharpoonright\overline{\chi})
\models^+  \psi\ 
\ \ \ \text{ or }\ \ \ (P_{\alpha}\upharpoonright\overline{\chi})\ =\ \emptyset\ \bigr)\\
P_{\alpha}\models^+ (\psi\vee\chi)
&\text{ iff }\
\bigl(\ (P_{\alpha}\upharpoonright\overline{\chi})
\models^+  \psi\ 
\ \ \ \text{ or }\ \ \ (P_{\alpha}\upharpoonright\overline{\chi})\ =\ \emptyset\ \bigr)\\
P_{\alpha}\models^+ \neg\psi
&\text{ iff }\ P_{\alpha}\models^- \psi\\
P_{\alpha}\models^+ \Diamond\varphi
&\text{ iff }\ P_{\alpha - 1}\models^+\varphi\text{ for some }
P_{\alpha - 1}\in P_{\alpha}
\\
P_{\alpha}\models^- (\chi \wedge \psi)
&\text{ iff }\ P_{\alpha}\models^- \chi
\text{ or }P_{\alpha}\models^- \psi\\
P_{\alpha}\models^- (\psi \wedge \chi)
&\text{ iff }\ P_{\alpha}\models^- \psi
\text{ or }P_{\alpha}\models^- \chi\\
P_{\alpha}\models^- (\chi \vee \psi)
&\text{ iff }\
\bigl(\ (P_{\alpha}\upharpoonright\overline{\chi})
\models^-  \psi\ 
\ \ \ \text{ and }\ \ \ (P_{\alpha}\upharpoonright\overline{\chi})\ \not=\ \emptyset\ \bigr)\\
P_{\alpha}\models^- (\psi\vee\chi)
&\text{ iff }\
\bigl(\ (P_{\alpha}\upharpoonright\overline{\chi})
\models^-  \psi\ 
\ \ \ \text{ and }\ \ \ (P_{\alpha}\upharpoonright\overline{\chi})\ \not=\ \emptyset\ \bigr)\\
P_{\alpha}\models^- \neg\psi
&\text{ iff }\ P_{\alpha}\models^+ \psi\\
P_{\alpha}\models^- \Diamond\varphi
&\text{ iff }\ P_{\alpha - 1}\models^-\varphi\text{ for all } P_{\alpha - 1}\in P_{\alpha}\\
\end{array}
\]
Assume then that $r(\chi) = r(\psi) = \alpha$. Then
\[
\begin{array}{ll}
P_{\alpha}\models^+ (\chi \wedge \psi)
&\text{ iff }\ P_{\alpha}\models^+ \chi
\text{ and }P_{\alpha}\models^+ \psi\\
P_{\alpha}\models^- (\chi\wedge\psi)
&\text{ iff }\ P_{\alpha}\models^- \chi
\text{ or } P_{\alpha}\models^- \psi\\
P_{\alpha}\models^+ (\chi \vee \psi)
&\text{ iff }\ P_{\alpha}\models^+ \chi
\text{ or }P_{\alpha}\models^+ \psi\\
P_{\alpha}\models^- (\chi\vee\psi)
&\text{ iff }\ P_{\alpha}\models^- \chi
\text{ and } P_{\alpha}\models^- \psi\\
\end{array}
\]

Now add an implication $\rightarrow$ to then language,
with $r(\varphi\rightarrow\psi)=r(\varphi\vee\psi)$. The semantics of $\chi\rightarrow\psi$ in
relation to $P_{\alpha}$ is defined as for other connectives when $r(\varphi\vee\psi) < \alpha$.
For the case where $r(\chi) < \alpha = r(\psi)$, it is quite natural to define
that $P_{\alpha}\models^+ (\chi \rightarrow \psi)$
iff $(P_{\alpha}\upharpoonright{\chi})
\models^+  \psi$ or $(P_{\alpha}\upharpoonright{\chi})\ =\ \emptyset$.
The negative clause would be $P_{\alpha}\models^- (\chi \rightarrow \psi)$
iff $(P_{\alpha}\upharpoonright{\chi})
\models^-  \psi$ and $(P_{\alpha}\upharpoonright{\chi})\ \not=\ \emptyset$.
For the case $r(\chi) = \alpha$, we put $P_{\alpha}\models^+ (\chi \rightarrow \psi)$
iff $P_{\alpha}\models^- \chi$ or 
$P_{\alpha}\models^+  \psi$ and also $P_{\alpha}\models^- (\chi \rightarrow \psi)$
iff $P_{\alpha}\models^+ \chi$ and 
$P_{\alpha}\models^-  \psi$. It is not entirely unnatural to define some kind of a 
diamond (i.e., $\diamond\varphi$) by the formula $\neg(\varphi \rightarrow \bot)$.
Here $\bot$ is $p\wedge\neg p$.

\subsubsection{First-order modal logic and beyond, a new take}\label{fossst2}

We then consider first-order modal logic and beyond. The approach is very similar to
the one for propositional modal logic above, but we
spell it out anyway, because we can.
Let $\mathcal{M}$ be a model set, i.e., a set of pairs $(\mathfrak{M},f)$ 
where $f$ is some assignment for $\mathfrak{M}$, interpreting
first-order variables.
Let $(P_\alpha)_{\alpha\in\mathbb{Z}_+}$ be a
sequence such that
\begin{enumerate}
\item
$P_1 \subseteq \mathcal{M}$.
\item
$P_{\alpha+1} \subseteq \mathcal{P}(P_{\alpha})$,
where $\mathcal{P}$ is again
the power set operator.
\end{enumerate}
We call $P_{\alpha}$ a
\emph{perspective} of rank $\alpha$. Again it often makes sense to 
not allow the empty set to belong to or be a perspective. Also, it is 
relatively natural to require the \emph{domain} of all the models in the
model set $\mathcal{M}$ to be the same. A perspective $P_{\alpha}$ is
called \emph{regular} if it is built from a model set where all models 
have the same domain. It is \emph{strongly regular} if we also have $P_{\alpha'}\not=\emptyset$
for all perspectives $P_{\alpha'}$ that are part of $P_{\alpha}$ on the 
different rank levels of $P_{\alpha}$ (with
the perspective $P_{\alpha}$ itself also nonempty).
Below we study strongly regular perspectives only.
Consider the language of first-order modal logic,
$$\varphi\ ::=\ Rx_1\dots x_k\ |\ (\varphi\wedge\varphi)\
|\ \ (\varphi\vee\varphi)\ |\ \neg\varphi\ |\
\exists x \varphi\ |\ \Diamond\varphi\ $$
where $Rx_1\dots x_k$ is any atom.
Define the \emph{rank} $r(\varphi)$ of a formula $\varphi$ as follows.
\begin{enumerate}
\item
$r(\chi) = 0$ for each atomic formula $\chi$.
\item
$r(\neg\chi) = r(\chi)$.
\item
$r(\chi\wedge\psi) = \mathit{max}\{r(\chi),r(\psi)\}$.
\item
$r(\chi\vee\psi) = \mathit{max}\{r(\chi),r(\psi)\}$.
\item
$r(\exists x \varphi) = r(\varphi)$.
\item
$r(\Diamond\varphi) = r(\varphi) + 1$.
\end{enumerate}
A perspective $P_{\alpha}$ of rank $\alpha$ can interpret
formulae of depth $\alpha$ or less with
the following semantics.
%

%
\begin{comment}
Assume first that $\alpha = 0$. Then
%
%
%
\[
\begin{array}{ll}
%
%
%
P_{0}\models^+ P\
%
&\text{ iff }\ (\mathfrak{M},w)\models P \text{ holds for all }
%
(\mathfrak{M},w)\in P_0\\
%
P_{0}\models^- P\
%
&\text{ iff }\ (\mathfrak{M},w)\not\models P \text{ holds for all }
%
(\mathfrak{M},w)\in P_0\\
%
\end{array}
\]
%
%
%
where $\models$ is the standard semantic turnstile of Kripke semantics.
%
\end{comment}
%

%
Assume first that $r(\varphi) < \alpha$. Then 
\[
\begin{array}{ll}
P_{\alpha}\models \varphi\
&\text{ iff }\ P_{\alpha - 1}\models \varphi \text{ for all }
P_{\alpha - 1}\in P_{\alpha}\\
\end{array}
\]
Note that if $\alpha = 1$, then $P_{\alpha - 1}$ is a 
model $(\mathfrak{M},f)$ in $P_{\alpha}$. Then the above
holds with the natural additional definition that 
$(\mathfrak{M},f) \models \varphi$
iff $(\mathfrak{M},f)\models_{\mathrm{FO}} \varphi$
where $\models_{\mathrm{FO}}$ is the
standard semantic turnstile of $\mathrm{FO}$.
Now, we define $(P_{\alpha}\upharpoonright\overline{\chi})$ analogously to
the definition given in Section \ref{prop2}. (In fact, the definition there can \emph{as it
stands} be read as a definition for model sets in predicate logic as well as 
propositional modal logic.)
Assume then that $r(\Diamond\varphi) = r(\psi)
= \alpha$ and $r(\chi) < \alpha$. Then 
\[
\begin{array}{ll}
P_{\alpha}\models^+ (\chi \wedge \psi)
&\text{ iff }\ P_{\alpha}\models^+ \chi
\text{ and }P_{\alpha}\models^+ \psi\\
P_{\alpha}\models^+ (\psi\wedge\chi)
&\text{ iff }\ P_{\alpha}\models^+ \psi
\text{ and }P_{\alpha}\models^+ \chi\\
P_{\alpha}\models^+ (\chi \vee \psi)
&\text{ iff }\
\bigl(\ (P_{\alpha}\upharpoonright\overline{\chi})
\models^+  \psi\ 
\ \ \ \text{ or }\ \ \ (P_{\alpha}\upharpoonright\overline{\chi})\ =\ \emptyset\ \bigr)\\
P_{\alpha}\models^+ (\psi\vee\chi)
&\text{ iff }\
\bigl(\ (P_{\alpha}\upharpoonright\overline{\chi})
\models^+  \psi\ 
\ \ \ \text{ or }\ \ \ (P_{\alpha}\upharpoonright\overline{\chi})\ =\ \emptyset\ \bigr)\\
P_{\alpha}\models^+ \neg\psi
&\text{ iff }\ P_{\alpha}\models^- \psi\\
P_{\alpha}\models^+ \Diamond\varphi
&\text{ iff }\ P_{\alpha - 1}\models^+\varphi\text{ for some }
P_{\alpha - 1}\in P_{\alpha}
\\
P_{\alpha}\models^- (\chi \wedge \psi)
&\text{ iff }\ P_{\alpha}\models^- \chi
\text{ or }P_{\alpha}\models^- \psi\\
P_{\alpha}\models^- (\psi \wedge \chi)
&\text{ iff }\ P_{\alpha}\models^- \psi
\text{ or }P_{\alpha}\models^- \chi\\
P_{\alpha}\models^- (\chi \vee \psi)
&\text{ iff }\
\bigl(\ (P_{\alpha}\upharpoonright\overline{\chi})
\models^-  \psi\ 
\ \ \ \text{ and }\ \ \ (P_{\alpha}\upharpoonright\overline{\chi})\ \not=\ \emptyset\ \bigr)\\
P_{\alpha}\models^- (\psi\vee\chi)
&\text{ iff }\
\bigl(\ (P_{\alpha}\upharpoonright\overline{\chi})
\models^-  \psi\ 
\ \ \ \text{ and }\ \ \ (P_{\alpha}\upharpoonright\overline{\chi})\ \not=\ \emptyset\ \bigr)\\
P_{\alpha}\models^- \neg\psi
&\text{ iff }\ P_{\alpha}\models^+ \psi\\
P_{\alpha}\models^- \Diamond\varphi
&\text{ iff }\ P_{\alpha - 1}\models^-\varphi\text{ for all } P_{\alpha - 1}\in P_{\alpha}\\
\end{array}
\]
Assume then that $r(\chi) = r(\psi) = \alpha$. Then
\[
\begin{array}{ll}
P_{\alpha}\models^+ (\chi \wedge \psi)
&\text{ iff }\ P_{\alpha}\models^+ \chi
\text{ and }P_{\alpha}\models^+ \psi\\
P_{\alpha}\models^- (\chi\wedge\psi)
&\text{ iff }\ P_{\alpha}\models^- \chi
\text{ or } P_{\alpha}\models^- \psi\\
P_{\alpha}\models^+ (\chi \vee \psi)
&\text{ iff }\ P_{\alpha}\models^+ \chi
\text{ or }P_{\alpha}\models^+ \psi\\
P_{\alpha}\models^- (\chi\vee\psi)
&\text{ iff }\ P_{\alpha}\models^- \chi
\text{ and } P_{\alpha}\models^- \psi\\
\end{array}
\]
%

\begin{comment}
To define a semantics for $\exists x$, we first give some auxiliary definitions.
The \emph{common domain} of an individual model $(\mathfrak{M},f)$ is 
simply the domain of $\mathfrak{M}$. The \emph{common domain} of a
perspective $P_{\alpha}$ is the intersection of the common domains of
all $P \in P_{\alpha}$. Here we assume $P_{\alpha}\not=\emptyset$.
\end{comment}
If $P$ is a perspective,
we let $P[a / x]$ denote the perspective where each model $(\mathfrak{M},f)$
existing on the rank $0$ level of $P$ is replaced by $(\mathfrak{M},f[a/x])$.
By the \emph{model domain} of $P$ we
refer to the domain of the models the (regular) perpective $P$ is built with.
%We define $(\mathfrak{M},f)[a/x] = (\mathfrak{M},f[a/x])$.
%
%
%
\begin{comment}
First, a \emph{choice function} of rank $\alpha > 1$ for a
perspective $\mathcal{P}$ of rank $\alpha$ is a
function $F: \mathcal{P}\rightarrow \mathcal{F}$, where $\mathcal{F}$ is
the set of choice functions of rank $\alpha - 1$. 
%
A choice function of rank $1$ is simply a choice 
function as defined in the previous sections. Thus a choice function of rank $1$
has model sets as domains. A choice function $F$ of rank $\alpha > 1$ has a
perspective $\mathcal{P}$ of rank $\alpha$ as a domain and it outputs
elements in the domains of the models that can be found from level $0$ of
the input perspective.
%

%
If $F$ is a choice function of rank $\alpha > 1$ and with domain $\mathcal{P}$,
we let $\mathcal{P}[F/x]$ denote
%
the perspective obtained from $\mathcal{P}$ by 
replacing each $\mathcal{P'}\in\mathcal{P}$ by $\mathcal{P}'[F'/x]$, 
where $F'$ is a choice function with domain $\mathcal{P}'$ mapping 
each input to $F(\mathcal{P}')$.
This ultimately gets down to the definition given for model sets. To be
properly defined, the
\end{comment}
%
%
%
Assume $\exists x\varphi$ and $P$ have the same rank.
We define
\begin{multline*}P\models^+ \exists x \varphi\text{ iff }\
{P}[a/x]
\models^+\varphi\
\text{for some $a$ in the model domain of $P$.}
\end{multline*}
and
\begin{multline*}P\models^- \exists x \varphi\text{ iff }\
{P}[a/x]
\models^-\varphi\
\text{for all $a$ in the model domain of $P$.}
\end{multline*}
%
%
%
%
%
%

%\newpage

The implication can of course be added to the picture with the same semantics as
given above for propositional modal logic.

Now consider the more general language 
$$\varphi\ ::=\ Rx_1\dots x_k\ |\ (\varphi\vee\varphi)\ |\ (\varphi\wedge\varphi)\ |\ \neg\varphi\ |\
Q x\varphi\ |\ \langle Q \rangle\varphi\ $$
where each $Q$ belongs to some symbol set $\mathcal{Q}$.
Both $Qx$ and $\langle Q\rangle$ will be associated 
with generalized quantifiers, so the language 
corresponding to $\mathcal{Q}$ will contain a 
generalized quantifier\footnote{Or a 
\emph{minor quantifier} \cite{double}, to be exact.} $Qx$ and a
generalized modality $\langle Q \rangle$
for each symbol $Q\in \mathcal{Q}$. The rank of a formula is measured in
terms of the possibly different operators $\langle Q\rangle$,
each application adding to the rank.

\begin{comment}
Let $(i,\dots , i_n)$ be a tuple of positive 
integers. Recall that a generalized quantifier of type $(i_1,\dots , i_n)$ is a 
class $Q$ of structures $(A, B_1,\dots , B_n)$ such that 
the following conditions hold.
%
%
%
\begin{enumerate}
%
\item
$A\not=\emptyset$
%
\item
For each $k\in\{1,\dots , n\}$, we have $B_k\subseteq A^{i_k}$
%
\item
$Q$ is closed under isomorphism, that it, if $f:A\rightarrow A'$ is 
an isomorphism from $(A, B_1,\dots , B_n)\in Q$ to $(A', B_1',\dots , B_n')$,
then $(A', B_1',\dots , B_n')\in\mathcal{C}$.
%
\end{enumerate}
\end{comment}
%

Recall here the definition of a unary generalized quantifier
(see \cite{double} for the (standard) definition we shall use).
Let $U$ be a unary generalized quantifier. We
let $\overline{U}$ denote $\{ (A,S)\, |\, (A,S)\not\in U\, \}$.

Let $U$ be a unary generalized quantifier.
Consider a class $\mathcal{C}$ of
structures $(A, B^+,B^-)$ such that 
the following conditions hold.
\begin{enumerate}
\item
$A\not=\emptyset$
\item
$B^+,B^-\subseteq A$
\item
$B^+\cap B^-\ = \emptyset$
\item
$\mathcal{C}$ is closed under isomorphism, that is, if $f:A\rightarrow A'$ is 
an isomorphism from $(A, B^+,B^-)\in \mathcal{C}$ to $(A, C^+,C^-)$
then $(A, C^+,C^-)\in \mathcal{C}$.
\item
For each $(A, B^+,B^-)\in \mathcal{C}$, there is a pair $(A,H)\in U$ 
such that $B^+\subseteq H$ and $B^-\subseteq A\setminus H$.
\item
For each $(A, B^+,B^-)\in \mathcal{C}$, there does \emph{not}
exists a pair $(A,H)\in \overline{U}$ 
such that $B^+\subseteq H$ and $B^-\subseteq A\setminus H$.
\item
For each $(A,H)\in U$, there exists a pair $(A,B^+,B^-)\in \mathcal{C}$ 
such that $B^+\subseteq H$ and $B^-\subseteq A\setminus H$.
\end{enumerate}
Then we say that $\mathcal{C}$ \emph{witnesses} $U$.
A \emph{minor quantifier} based on $U$ is a pair $Q = (\mathcal{C},\mathcal{D})$ such
that $\mathcal{C}$ witnesses $U$ and $\mathcal{D}$ witnesses $\overline{U}$.
We denote $\mathcal{C}$ by $Q_+$ and $\mathcal{D}$ by $Q_-$.
We fix the lastly defined semantics for modal predicate logic given above (in this 
Section), with only the 
clauses for $Q x$ and $\langle Q\rangle$ redifined.
Recall that the semantics for predicate logic with quantifiers $Q x$ on ordinary 
first-order models is defined, e.g., in \cite{double}.
Suppose $\langle Q\rangle\varphi$ and $P$ have the same rank. We define that 
$$P_{\alpha}\models^+\langle Q\rangle\varphi
\text{ iff } (P_{\alpha},P^+,P^-)\in Q_+$$
where $P^+\subseteq P_{\alpha}$ is some
set of perspectives $P_{\alpha-1}\in P_{\alpha}$
such that $P_{\alpha-1}\models^+\varphi$
and $P^-\subseteq P_{\alpha}$ is some set of
perspectives $Q_{\alpha-1}\in P_{\alpha}$
such that $Q_{\alpha-1}\models^-\varphi$.
Similarly,
$$P_{\alpha}\models^-\langle Q\rangle\varphi
\text{ iff } (P_{\alpha},P^+,P^-)\in Q_-$$
where again $P^+\subseteq P_{\alpha}$ is some 
set of perspectives $P_{\alpha-1}\in P_{\alpha}$
such that $P_{\alpha-1}\models^+\varphi$
and $P^-\subseteq P_{\alpha}$ some set of
perspectives $Q_{\alpha-1}\in P_{\alpha}$
such that $Q_{\alpha-1}\models^-\varphi$.
Now let $M$ denote the model domain of $P$. We define
\begin{multline*}P\models^+ Qx \varphi\text{ iff }
\text{for some $(M,S,T)\in Q_+$, we have}\\
{P}[a/x]
\models^+\varphi\text{ for all }a\in S\text{ and}\\
{P}[a/x]
\models^-\varphi\text{ for all }a\in T.\\
\end{multline*}
Also,
\begin{multline*}P\models^- Qx \varphi\text{ iff }
\text{for some $(M,S,T)\in Q_-$, we have}\\
{P}[a/x]
\models^+\varphi\text{ for all }a\in S\text{ and}\\
{P}[a/x]
\models^-\varphi\text{ for all }a\in T.\\
\end{multline*}

It is easy to see how to generalize all this to formulae with more
general generalized (minor) quantifiers. For example, $\langle Q\rangle(\varphi,\psi)$
with $Q$ indicating that more elements of the 
perspective satisfy $\varphi$ than $\psi$, is interesting.

%\newpage

%
\subsection{More on perspectives}
It is natural to define an implication equivalent to $\Box(\neg \varphi\vee \psi)$,
where $\Box$ is $\neg\Diamond\neg$. Also, it is possible to 
consider the approach $P\models\varphi\Rightarrow \psi$ iff
$P'\models \psi$, where $P'\subseteq P$ contains those $Q\in P$
such that $Q\models\varphi$. Note that $\Rightarrow$ will not
alter the rank but the implication with $\Box(\neg \varphi\vee \psi)$ will.

It is also very natural to go
multimodal. This is easy by replacing perspectives by pairs with a 
perspective and a label. The label denotes the agent (or whatever). The 
perspective itself can contain sets with different labels, but the notion of 
rank of course has to be adjusted. It is not difficult to model these 
approaches in Kripke semantics (by creating sets, i.e., perspectives, with
the help of $\Box$). 
Suppose $x$ has some value in $\mathbb{N}$.
Consider the following reasoning scenario. If $x$ is even, then it is not 
possible that $x$ is odd. Thus, if $x$ is even, it is not the 
case that it is possible that $x$ is odd and it is possible that $x$ is even.
Writing this in symbols (without any particular semantics fixed), we get
``$x\text{ is even}\rightarrow\neg(\Diamond(x\text{ is odd})\wedge \Diamond(x\text{ is even}))$''. We symmetrically conclude that
``$x\text{ is odd}\rightarrow\neg(\Diamond(x\text{ is odd})\wedge \Diamond(x\text{ is even}))$''.
Let us abbreviate these statements by $\varphi_{even}\rightarrow \neg \chi$ and
$\varphi_{odd}\rightarrow \neg \chi$. Now, supposing we 
can deduce $C$ from $A\vee B$, $A\rightarrow C$, $B\rightarrow C$, we
deduce $\neg \chi$ from $\varphi_{even}\vee\varphi_{odd}$.
That is, we deduce $\neg(\Diamond(x\text{ is odd})\wedge \Diamond(x\text{ is even}))$ because
surely $x$ is even or odd.
However, it is clear that $(\Diamond(x\text{ is odd})\wedge \Diamond(x\text{ is even}))$ holds.
One way to reject the faulty 
reasoning is to assert that epistemic and metaphysical modalities are mixed up.
Indeed, we can reject ``$x\text{ is even}\rightarrow\neg(\Diamond(x\text{ is odd}))$''
if we read $\Diamond$ as ``it appears to be possible that.'' But if we
only wish to consider possible scenarios without there being an 
actual world (which would be unknown but would give $x$
some value), we can reject the 
deduction rule ``$C$ from $A\vee B$, $A\rightarrow C$, $B\rightarrow C$,'' as
the above ssemantics does if $\rightarrow$ is 
given the interpretation of $\Rightarrow$ we defined
above. This is because $A\vee B$ does not 
mean that $A$ is surely true or $B$ is surely true, simply that the space of 
scenarios splits so that $A$ is the case in the 
first and $B$ in the second scenario.
\bibliographystyle{plain}
\bibliography{aaabibfile}

\end{document}